\documentclass[12pt]{amsart}
\usepackage{amsfonts,amscd,latexsym,amsmath,amssymb,cancel,enumerate,mathrsfs,color,hyperref}
\theoremstyle{plain}
\newtheorem{master}{Master}[section]
\newtheorem{prop}[master]{Proposition}
\newtheorem{thm}[master]{Theorem}
\newtheorem{fact}[master]{Fact}
\newtheorem{lem}[master]{Lemma}

\newtheorem{claim}[master]{Claim}
\theoremstyle{definition}
\newtheorem{defin}[master]{Definition}

\theoremstyle{remark}
\newtheorem{remark}[master]{Remark}
\numberwithin{equation}{section}

\newcommand{\Rea}{\mathbb{R}}
\newcommand{\Nat}{\mathbb{N}}
\newcommand{\Rat}{\mathbb{Q}}
\newcommand{\Int}{\mathbb{Z}}

\begin{document}
\title{Metrically universal abelian groups}
\author{Michal Doucha}
\address{Institute of Mathematics, Academy of Sciences, Prague, Czech republic}
\address{Institute of Mathematics, Polish Academy of Sciences, Warsaw, Poland}
\email{doucha@math.cas.cz}
\thanks{The research of the author was partially supported by the grant IAA100190902 of Grant Agency of the Academy of Sciences
of the Czech Republic and by funds allocated to the implementation of the international co-funded project in the years 2014-2018, 3038/7.PR/2014/2, and by the EU grant PCOFUND-GA-2012-600415.}
\begin{abstract}
We give a positive answer to the question of Shkarin (\emph{On universal abelian topological groups}, Mat. Sb.  190  (1999), no. 7, 127-144) whether there exists a metrically universal abelian separable group equipped with invariant metric.

Our construction also gives an example of a group structure on the Urysohn universal space that is substantially different from the previously known examples.

Under some cardinal arithmetic assumptions, our results generalize to higher cardinalities.
\end{abstract}
\maketitle
\section*{Introduction}
In \cite{Sk}, S. A. Shkarin constructed a separable abelian metric group $G$ that is topologically universal for the class of separable metrizable abelian groups. It means that for every abelian separable metrizable topological group $H$ there exists $\iota:H\hookrightarrow G$ which is both a group monomorphism and a topological embedding. This group was further investigated by P. Niemiec in \cite{Nie1}. Problem 1 in the Shkarin's paper is whether there exists a separable abelian group with an invariant metric which is metrically universal for the class of all separable abelian groups equipped with invariant metric, i.e. whether there exists a separable abelian metric group $(G,d)$, $d$ is invariant, such that for any other separable abelian metric group $(H,p)$, $p$ being invariant, there is $\iota:H\hookrightarrow G$ which is both a group monomorphism and an isometric map. We provide a positive answer to this question here. Thus the following is the main result of this paper.
\begin{thm}\label{main_theorem}
There exists a separable abelian metric group $(\mathbb{G},d)$, where $d$ is invariant, such that for any separable abelian metric group $(H,p)$, with $p$ invariant, there is a subgroup $H'\leq \mathbb{G}$ and an isometric isomorphism $\iota: H\rightarrow H'$.
\end{thm}
Another interesting phenomenon investigated recently is the group structures on the Urysohn universal space. In \cite{CaVe}, P. Cameron and A. Vershik prove that there exists an abelian group structure on the Urysohn space which is monothetic, i.e. it contains a dense cyclic subgroup. Niemiec in \cite{Nie2} proves that there is a structure of a Boolean group on the Urysohn space and that this group is moreover metrically universal for the class of separable metric abelian Boolean groups. Further, Niemiec in \cite{Nie1} proves that the Shkarin's group is also isometric to the Urysohn space. However, it remains open and stated as an open question there whether the Shkarin's group is monothetic, even whether it is different from the type of groups considered by Cameron and Vershik. Here we prove that the group $\mathbb{G}$ from Theorem \ref{main_theorem} is also isometric to the Urysohn universal space and moreover that it is isometrically isomorphic neither to the Shkarin/Niemiec's group nor to the Cameron-Vershik groups.
\section{Definitions and preliminaries}
\begin{defin}
Let $(G,d)$ be a metric group, i.e. group equipped with a metric so that the group operations are continuous in the topology given by the metric. We say that the metric $d$ is finitely generated if there exists a finite set $A\subseteq G^2$ (called generating set for the metric) such that for every $a,b\in G$ we have $d(a,b)=\min\{d(a_1,b_1)+\ldots+d(a_n,b_n):n\in \Nat,(a_i,b_i)\in A\wedge a=a_1\cdot\ldots\cdot a_n,b=b_1\cdot\ldots\cdot b_n,i\leq n\}$. In particular, we assume that for any $a,b\in G$ there exist $(a_1,b_1),\ldots,(a_n,b_n)\in A$ such that $a=a_1\cdot\ldots\cdot a_n$, $b=b_1\cdot\ldots\cdot b_n$, thus $G$ must be finitely generated.
\end{defin}
\begin{lem}\label{greatestmet_lem}
Let $d$ be a finitely generated metric on some group $G$ which is generated by the set $A\subseteq G^2$. Then $d$ is two-sided invariant. Moreover, for any other two-sided invariant metric $p$ on $G$ which agrees with $d$ on $A$ we have $p\leq d$.
\end{lem}
\begin{proof}
Recall that a metric $d$ on a group $G$ is two-sided invariant iff for every $x,y,u,v\in G$ we have $d(x\cdot y,u\cdot v)\leq d(x,u)+d(y,v)$. It is clear that finitely generated metrics satisfy this property.

Moreover, if $p$ is any other two-sided invariant metric on $G$ which agrees with $d$ on $A$, then because of invariance of $p$, for every $a,b\in G$ and $(a_1,b_1),\ldots,(a_n,b_n)\in A$ such that $a=a_1\cdot\ldots\cdot a_n$, $b=b_1\cdot\ldots\cdot b_n$ we must have $p(a,b)\leq p(a_1,b_1)+\ldots+p(a_n,b_n)$, thus $p(a,b)\leq d(a,b)$.
\end{proof}
The following simple fact will be used without explicit referring through the rest of the paper.
\begin{fact}
Let $(G,d)$ be a group with a metric finitely generated by the set $A\subseteq G^2$. Let $a,b\in G$ be arbitrary and let $(a_1,b_1),\ldots,(a_n,b_n)\in A$ be such that $a=a_1\cdot\ldots\cdot a_n$, $b=b_1\cdot\ldots\cdot b_n$ and $d(a,b)=d(a_1,b_1)+\ldots+d(a_n,b_n)$. Then for any $1\leq i\leq j\leq n$ we have $d(a_i\cdot\ldots\cdot a_j,b_i\cdot\ldots \cdot b_j)=d(a_i,b_i)+\ldots+d(a_j,b_j)$.
\end{fact}
In what follows we shall consider finitely generated abelian groups with finitely generated metrics. In particular, we shall use the additive notation for the group operations and the neutral element. 
\begin{remark}\label{defin_remark}
In practice, when defining a finitely generated metric on some group $G$ we prescribe the distances on some finite set $A\subseteq G^2$ and then for any $a,b\in G$ we put $$d(a,b)=\min\{d(a_1,b_1)+\ldots+d(a_n,b_n):a=a_1+\ldots+a_n,$$ $$b=b_1+\ldots+b_n,(a_i,b_i)\in A,i\leq n\}.$$ We will not assume that $A$ is minimal possible, i.e. we will allow that for some $(a,b)\in A$ we have $d(a,b)=d(a_1,b_1)+\ldots+d(a_n,b_n)$ for some $(a_1,b_1),\ldots,(a_n,b_n)\in A$ such that $a=a_1+\ldots+a_n$ and $b=b_1+\ldots+b_n$. However, we will always assume that the prescription is consistent in the sense that for every $(a,b)\in A$ and any $(a_1,b_1),\ldots,(a_n,b_n)\in A$ such that $a=a_1+\ldots+a_n$ and $b=b_1+\ldots+b_n$ we really have $d(a,b)\leq d(a_1,b_1)+\ldots+d(a_n,b_n)$.

This problem will occur later in the text (in the proof of Claim \ref{claim2}) when we will have finitely generated metrics on some group $G_1$ and $G_2$ with generating sets $A_1\subseteq G_1^2$ and $A_2\subseteq G_2^2$. We will want to define a finitely generated metric on $G_1\oplus G_2$ so that the set $A=A_1\cup A_2\cup \{(a_i,b_i):1\leq i\leq n\}$, where for each $i\leq n$, $a_i\in G_1$ and $b_i\in G_2$, will be generating. For a pair $(a,b)\in A_i$ we naturally set $d_{G_1\oplus G_2}(a,b)=d_{G_i}(a,b)$ for $i\in \{1,2\}$. However, now the consistency requirement will not be obvious (and necessarilly true) and we will have to prove it.

Finally, let us mention a situation where the consistency requirement is always satisfied. Suppose that $G$ is an abelian group with invariant metric $d$ and $A\subseteq G$ is a finite generating subset. Let $\rho$ be an invariant metric on $G$ generated by values of $d$ on $A^2$, i.e. for every $a,b\in G$ we have $\rho(a,b)=\min\{d(a_1,b_1)+\ldots+d(a_n,b_n):n\in \Nat,(a_i,b_i)\in A^2\wedge a=a_1\cdot\ldots\cdot a_n,b=b_1\cdot\ldots\cdot b_n,i\leq n\}$. Then clearly for each $a,b\in A$ we do have $\rho(a,b)=d(a,b)$.\\
\end{remark}

We add here a brief informal introduction to the Fra\" iss\' e theory that will be used in this paper. By no means we try to be as general as possible. We present just the basic theory almost as it appeared already in 1950's in the work of Fra\" iss\' e (see \cite{Fr}), and we shall only treat Fra\" iss\' e theory for groups. We refer to (Chapter 7, \cite{Ho}) for a detailed exposition of this classical Fra\" iss\' e theory that is very similar to the approach we take here (the reader is encouraged to look at Exercise 11 of Chapter 7 in \cite{Ho} as that is basically the version of the Fra\" iss\' e theorem we need). For a very general category-theoretic approach to Fra\" iss\' e theory we refer the reader to \cite{Kub}.

Consider some class of groups which is closed under taking direct limits. Let $\mathcal{K}$ be some countable subclass and let $\mathcal{M}$ denote a set of monomorphisms between groups from $\mathcal{K}$. For two groups $A,B\in \mathcal{K}$ we write $A\sqsubseteq B$ if there exists a monomorphism from $A$ into $B$ that belongs to $\mathcal{M}$.

We say that $\mathcal{K}$ satisfies the $\sqsubseteq$-amalgamation property if whenever we have $A,B,C\in \mathcal{K}$ such that $A\sqsubseteq B$ is witnessed by some monomorphism $\phi_B\in \mathcal{M}$ and $A\sqsubseteq C$ is witnessed by some monomorphism $\phi_C\in \mathcal{M}$, then there exists $D\in \mathcal{K}$ such that $B\sqsubseteq D$ and $C\sqsubseteq D$ are witnessed by monomorphisms $\psi_B\in \mathcal{M}$, resp. $\psi_C\in \mathcal{M}$ such that $\psi_C\circ \phi_C=\psi_B\circ \phi_B$.

We shall call such $(\mathcal{K},\mathcal{M},\sqsubseteq)$ a $\sqsubseteq$-Fra\" iss\' e class (of groups). We shall usually omit the set $\mathcal{M}$ from the notation and write just $(\mathcal{K},\sqsubseteq)$.\\

Denote by $\bar{\mathcal{K}}$ the set of all direct limits of groups from $\mathcal{K}$. Thus if $\mathcal{K}$ is for instance the class of all finitely generated abelian groups, then $\bar{\mathcal{K}}$ will be the class of all at most countably generated abelian groups.

If $A\in\mathcal{K}$ and $B\in \bar{\mathcal{K}}$, then we write $A\sqsubseteq B$ if $B$ can be written as a direct limit $B_1\sqsubseteq B_2\sqsubseteq\ldots$ and there is some $m$ such that $A\sqsubseteq B_m$. We then have the following theorem.
\begin{thm}[Fra\" iss\' e theorem (for groups), see \cite{Ho} or \cite{Kub} ]\label{fraissethm}
Let $(\mathcal{K},\sqsubseteq)$ be a $\sqsubseteq$-Fra\" iss\' e class of groups. Then there exists a unique up to isomorphism group $K\in\bar{\mathcal{K}}$ with the following properties:
\begin{enumerate}
\item $K$ is a direct limit of some sequence $K_1\sqsubseteq K_2\sqsubseteq\ldots$, where $K_i$'s are from $\mathcal{K}$.

\item For any $A\in \mathcal{K}$ we have $A\sqsubseteq K$.
\item For any $A_1,A_2\in \mathcal{K}$ and embeddings $\iota_1:A_1\rightarrow K_n$, $\rho:A_1\rightarrow A_2$, for some $n$ and where $\iota_1,\rho\in \mathcal{M}$, there exists $m>n$ and an embedding $\iota_2: A_2\rightarrow K_m$ such that $\iota_2\circ \rho=\iota_1$.
\item If $A_1, A_2\in \mathcal{K}$ are isomorphic and we have $A_1,A_2\sqsubseteq K$ which is witnessed by embeddings $\iota_1,\iota_2\in \mathcal{M}$, then there exists an automorphism of $K$ sending $\iota_1[A_1]$ onto $\iota_2[A_2]$.

\end{enumerate}
\end{thm}

We call such $K$ the Fra\" iss\' e limit of $\mathcal{K}$.\\

Let us also make few notes concerning notations here. If $X$ is a metric space then by $d_X$ we denote its metric. However, in some cases when there is no danger of confusion we may just denote the metric as $d$. Similarly, if $G$ is an abelian group then by $0_G$ we denote the neutral element, and in some cases we may just denote it as $0$. Moreover, if $(G,d_G)$ is a metric group with neutral element $0_G$ then by $d_G(x)$, for every $x\in G$, we mean $d_G(x,0_G)$, the corresponding value on the group, i.e. the distance of $x$ from the neutral element.

We conclude this section by the following simple fact that will be used without referring in the paper.
\begin{fact}[Lemma 1 in \cite{Sk}]
Let $(G,d)$ be an abelian group equipped with an invariant metric. Then for any $g\in G$ the sequence $(\frac{d(n\cdot g)}{n})_{n\in \Nat}$ converges and we have $\lim_{n\to \infty} \frac{d(n\cdot g)}{n}=\inf_n \frac{d(n\cdot g)}{n}$.
\end{fact}
\section{The $\sqsubseteq$-Fra\" iss\' e class of groups}
In this section, we construct the metric group $(\mathbb{G},d)$ from Theorem \ref{main_theorem} and prove that it is isometric to the Urysohn space. We start with a proposition dealing with realizing Kat\v etov functions on metric abelian groups. Let us recall that a function $f:X\rightarrow \Rea^+$, where $X$ is some metric space and $\Rea^+$ denotes the set of all positive reals, is called Kat\v etov if it satisfies for every $x,y\in X$, $(|f(x)-f(y)|\leq d_X(x,y)\leq f(x)+f(y)$. This function is supposed to prescribe distances from some new imaginary point to the points of $X$. We refer to the paper \cite{Kat} of Kat\v etov for more details.
\begin{prop}\label{Katetov_realiz}
Let $(G,d)$ be an abelian group with invariant metric $d$. Let $A\subseteq G$ be some finite subset and let $f:A\rightarrow \Rea^+$ be a Kat\v etov function. Then there exists an invariant metric $\bar{d}$ on $G\oplus \Int$ that extends $d$ and if we denote the generator of the added copy of $\Int$ as $z$ then for every $a\in A$, $(\bar{d}(z,a)=f(a))$. In other words, the Kat\v etov function $f$ is realized as a distance function from $z$.

Moreover, if the metric $d$ on $G$ is rational-valued and finitely generated, and $f$ has values in the rationals, then the metric $\bar{d}$ can be made also rational-valued and finitely generated.
\end{prop}
\begin{proof}
We may suppose that $A$ is symmetric since if it were not we would take the least symmetric superset $A'$ of $A$ and extend $f$ to $f'$ on $A'$ arbitrarily so that it would still be a Kat\v etov function, e.g. for every $a\in A'$ we could set $f'(a)=\min\{d(a,b)+f(b):b\in A\}$. Consider $G\oplus \Int$ and denote the generator of this new copy of the integers $\{0_G\}\oplus \Int$ as $z$. We now define $\bar{d}$ on  $G\oplus \Int$. Because of the invariance of the metric it suffices to define the distance $\bar{d}$ between elements of the form $0_G+n\cdot z$ and $g+0\cdot z$ from $G\oplus \Int$, where $n\in \Int$ and $g\in G$. In order to simplify the notation, we shall write just $n\cdot z$, resp. just $g$, instead of formally correct $0_G+n\cdot z$, resp. $g+0\cdot z$. Thus let $n\in \Int$ and $g\in G$ be arbitrary, we set $$\bar{d}(n\cdot z,g)= \min\{d(g,h)+f(\varepsilon_1\cdot h_1)+\ldots+f(\varepsilon_m\cdot h_m):m\in \Nat\cup\{0\},$$ $$h\in G,h_1+\ldots+h_m=h, h_i\in A\wedge \varepsilon_i\in\{1,-1\}\forall i\leq m,\sum_{i=1}^m \varepsilon_i=n\}.$$

\begin{enumerate}
\item\label{item1} We need to check that if $n=0$ then $\bar{d}(n\cdot z,g)=d(g,0_G)$ for every $g\in G$.

It follows directly from the definition of $\bar{d}$ that we have $d(g,0_G)\geq \bar{d}(0\cdot z,g)$. To prove the reverse inequality we must show that for every $h\in G$, any $h_1,\ldots,h_m\in A$ such that $h_1+\ldots+h_m=h$ and any $\varepsilon_1,\ldots,\varepsilon_m\in \{1,-1\}$ such that $\varepsilon_1+\ldots+\varepsilon_m=0$ (note that $m$ is even) we have $d(g,h)+f(\varepsilon_1\cdot h_1)+\ldots+f(\varepsilon_m\cdot h_m)\geq d(g,0_G)$. Write $m$ as $2l$ and we proceed by induction on $l$.

If $l=0$ it is obvious. Suppose $l=1$. So we have some $h_1,h_2\in A$ and we must prove that $d(g,0_G)\leq d(g,h_1+h_2)+f(h_1)+f(-h_2)$. Since $f$ is Kat\v etov we have $f(h_1)+f(-h_2)\geq d(h_1,-h_2)=d(h_1+h_2,0_G)$. By triangle inequality, we immediately get $d(g,0_G)\leq d(g,h_1+h_2)+d(h_1+h_2,0_G)\leq d(g,h_1+h_2)+f(h_1)+f(-h_2)$.

Suppose now that $l>1$ and we have proved the claim for all values smaller. Then we take some $i,j\leq 2l$ such that $\varepsilon_i=1$ and $\varepsilon_j=-1$. Since $f$ is Kat\v etov we have $f(h_i)+f(-h_j)\geq d(h_i,-h_j)=d(h_i+h_j,0_G)=d(h-h_i-h_j,h)$. Thus we have $d(g,h)+f(\varepsilon_1\cdot h_1)+\ldots+f(\varepsilon_m\cdot h_m)\geq d(g,h)+d(h-h_i-h_j,h)+\sum_{k\in m\setminus\{i,j\}} f(\varepsilon_k\cdot h_k)\geq d(g,h-h_i-h_j)+\sum_{k\in m\setminus\{i,j\}} f(\varepsilon_k\cdot h_k)$ and we use the inductive assumption and we are done.
\item\label{item2} It remains to check that if $g\in A$ then $\bar{d}(z,g)=f(g)$.

The proof is similar to the previous one. Again, the inequality $\bar{d}(z,g)\leq f(g)$ follows directly from the definition of $\bar{d}$. To prove the reverse inequality, we must show that for every $h\in G$, any $h_1,\ldots,h_m\in A$ such that $h_1+\ldots+h_m=h$ and any $\varepsilon_1,\ldots,\varepsilon_m\in \{1,-1\}$ such that $\varepsilon_1+\ldots+\varepsilon_m=1$ (note that $m$ is now odd) we have $d(g,h)+f(\varepsilon_1\cdot h_1)+\ldots+f(\varepsilon_m\cdot h_m)\geq f(g)$. Write $m$ as $2l+1$ and we proceed by induction on $l$.

If $l=0$ then we have to prove that $d(g,h)+f(h)\geq f(g)$. However, this follows immediately since $f$ is Kat\v etov. Suppose now that $l>0$ and we have proved the claim for all smaller values. We again take some $i,j\leq 2l+1$ such that $\varepsilon_i=1$ and $\varepsilon_j=-1$. Since $f$ is Kat\v etov we have $f(h_i)+f(-h_j)\geq d(h_i,-h_j)=d(h_i+h_j,0_G)=d(h-h_i-h_j,h)$. Thus we have $d(g,h)+f(\varepsilon_1\cdot h_1)+\ldots+f(\varepsilon_m\cdot h_m)\geq d(g,h)+d(h-h_i-h_j,h)+\sum_{k\in m\setminus\{i,j\}} f(\varepsilon_k\cdot h_k)\geq d(g,h-h_i-h_j)+\sum_{k\in m\setminus\{i,j\}} f(\varepsilon_k\cdot h_k)$ and we use the inductive assumption and we are done.
\end{enumerate}

Now note that in case $d$ is rational-valued and finitely generated and $f$ has values in the rationals, $\bar{d}$ is indeed rational-valued and finitely generated. To see this, let $B\subseteq G^2$ be the generating set for $d$. Without loss of generality, we may assume that $A^2\subseteq B$. We now check that $\bar{d}$ is generated by the set $B\cup \{(\varepsilon\cdot z,\varepsilon\cdot a),(\varepsilon\cdot a,\varepsilon\cdot z):a\in A,\varepsilon\in\{-1,1\}\}$. For any $n\in \Int$ and $g\in G$ we need to check that the values $$d_1=\bar{d}(n\cdot z,g)=\min\{d(g,h)+f(\varepsilon_1\cdot h_1)+\ldots+f(\varepsilon_m\cdot h_m):m\in \Nat\cup\{0\},$$ $$h\in G,h_1+\ldots+h_m=h, h_i\in A\wedge \varepsilon_i\in\{1,-1\} \forall i\leq m,\sum_{i=1}^m \varepsilon_i=n\}$$ and $$d_2=\min\{\bar{d}(a_1,b_1)+\ldots+\bar{d}(a_m,b_m):n\cdot z=a_1+\ldots+a_m,$$ $$g=b_1+\ldots+b_m, (a_i,b_i)\in B\cup\{(\varepsilon\cdot z,\varepsilon\cdot a),(\varepsilon\cdot a,\varepsilon\cdot z):a\in A,\varepsilon\in\{-1,1\}\}\}$$ are the same.\\

First we show that $d_1\geq d_2$. Suppose that for some $m\in \Nat\cup\{0\},h\in G,h_1+\ldots+h_m=h$, $h_i\in A$ and $\varepsilon_i\in\{1,-1\}$ for $i\leq m$, and $\sum_{i=1}^m \varepsilon_i=n$, we have $d_1=d(g,h)+f(\varepsilon_1\cdot h_1)+\ldots+f(\varepsilon_m\cdot h_m)$. By \eqref{item1} and \eqref{item2} for any $a\in A$ and any $h,h'\in G$, we have $f(a)=\bar{d}(z,a)$ and $d(h,h')=\bar{d}(h,h')$.  Also by definition, $d(g,h)$ can be written as a sum $d(a_1,b_1)+\ldots+d(a_j,b_j)=\bar{d}(a_1,b_1)+\ldots+\bar{d}(a_j,b_j)$, where $a_1+\ldots+a_j=g$, $b_1+\ldots+b_j=h$ and $(a_i,b_i)\in B$ for all $i\leq j$. It follows that $d_1$ can be written as $\bar{d}(a_1,b_1)+\ldots+\bar{d}(a_j,b_j)+\bar{d}(z,\varepsilon_1\cdot h_1)+\ldots+\bar{d}(z,\varepsilon_m\cdot h_m)$ which is by definition greater or equal to $d_2$.\\ 

On the other hand, consider a sum of the form $\bar{d}(a_1,b_1)+\ldots+\bar{d}(a_m,b_m)$, where $n\cdot z=a_1+\ldots+a_m$, $g=b_1+\ldots+b_m$ and $(a_i,b_i)\in B\cup\{(\varepsilon\cdot z,\varepsilon\cdot a),(\varepsilon\cdot a,\varepsilon\cdot z):a\in A,\varepsilon\in\{-1,1\}\}$. We can split the set $\{1,\ldots,m\}$ into three disjoint subsets $I_1$, $I_2$ and $I_3$, where 
\begin{itemize}
\item $(a_i,b_i)$ is of the form $(\varepsilon_i\cdot z,h_i)$ for some $h_i\in A$ and $\varepsilon_i\in \{-1,1\}$, when $i\in I_1$,
\item $(a_i,b_i)$ is of the form $(a,\varepsilon\cdot z)$ for some $a\in A$ and $\varepsilon\in \{-1,1\}$, when $i\in I_2$,
\item and $(a_i,b_i)\in B$ when $i\in I_3$.

\end{itemize}
Suppose for a moment that the set $I_2$ is empty. We show that it is possible to rewrite the sum above in the form $d(g,h)+f(\varepsilon_1\cdot h_1)+\ldots+f(\varepsilon_k\cdot h_k)$ for some appropriate values and thus $d_2\geq d_1$. Since $\sum_{i\leq m} a_i=n\cdot z$ and also $\sum_{i\in I_1} a_i=n\cdot z$, we get that $\sum_{i\in I_3} a_i=0_G$. Denote by $h$ the sum $\sum_{i\in I_1} b_i$ and also by $s$ the sum $\sum_{i\in I_3} b_i$. We have $h+s=g$, thus $s=g-h$. Since $\sum_{i\in I_3} \bar{d}(a_i,b_i)=d(0,g-h)=d(g,h)$ and for each $i\in I_1$ we have $\bar{d}(a_i,b_i)=f(\varepsilon_i\cdot h_i)$, we indeed get that $\sum_{i\leq m} \bar{d}(a_i,b_i)=d(g,h)+\sum_{i\in I_1} f(\varepsilon_i\cdot h_i)$.

We now show that we may suppose that $I_2$ is empty. Since $b_1+\ldots+b_m=g\in G$ we must have that $\sum_{i\in I_2} b_i=0$. Consider $\bar{a}=\sum_{i\in I_2} a_i\in G$. By \eqref{item1}, $\sum_{i\in I_2} \bar{d}(a_i,b_i)\geq d(\bar{a},0)$. Since there exist $(a'_1,b'_1),\ldots,(a'_l,b'_l)\in B$ such that $a'_1+\ldots+a'_l=\bar{a}$, $b'_1+\ldots+b'_l=0$ and $d(\bar{a},0)=d(a'_1,b'_1)+\ldots+d(a'_l,b'_l)$, we can replace the subsequence $\{(a_i,b_i):i\in I_2\}$ by $(a'_1,b'_1),\ldots,(a'_l,b'_l)$ and we are done.\\

This finishes the proof of the proposition.
\end{proof}
The following proposition is the analog of Theorem 2.12 in \cite{Nie1}. It will be crucial later in the embedding construction.
\begin{prop}\label{sums_dense}
Let $(G,d_G)$ be an abelian group with invariant metric $d_G$ which is of density $\kappa$. Then there exists an abelian supergroup $(H,d_H)$ with invariant metric which is of density $\kappa\times \aleph_0$ such that $d_H\upharpoonright G=d_G\upharpoonright G$ and $H$ has a dense subgroup of the form $\bigoplus_{\kappa\times \aleph_0} \Int$, i.e. the free abelian group of $\kappa\times \aleph_0$-many generators.
\end{prop}
\begin{proof}
Let $G'\leq G$ be a dense subgroup of size $\kappa$. In the following, we will not distinguish between cardinals and the corresponding ordinals, i.e. the least ordinals of that cardinality.

Enumerate $G'$ as $\{g_\alpha:\alpha<\kappa\}$. Using Proposition \ref{Katetov_realiz}  and transfinite induction of length $\kappa\times \aleph_0$ we produce a group $G\oplus (\bigoplus_{\alpha<\kappa\times \aleph_0} \Int_\alpha)$ where for each $\alpha<\kappa\times \aleph_0$, $\Int_\alpha$ is a copy of $\Int$ generated by $z_\alpha$ and for every $\beta<\kappa$ and every $m\in \Nat$ there is $\alpha$ such that $z_\alpha$ realizes the Kat\v etov function $f_\beta^m$ defined on $\{g_\beta\}$ such that $f_\beta^m(g_\beta)=1/m$. It follows that $\bigoplus_{\alpha<\kappa\times \aleph_0} \Int_\alpha$ is a dense subgroup of cardinality $\kappa\times \aleph_0$ in $G\oplus (\bigoplus_{\alpha<\kappa\times \aleph_0} \Int_\alpha)=H$ and we are done.
\end{proof}
We are now ready to construct the group $\mathbb{G}$. The group is constructed as the metric completion of a certain (generalized) Fra\" iss\' e limit. Thus we start by defining the class which we then show to be Fra\" iss\' e.
\begin{defin}[$\sqsubseteq$-Fra\" iss\' e class of finite direct sums of $\Int$ with finitely generated metrics]
Elements of the class $\mathcal{G}$ are finite direct sums (direct products equivalently) of $\Int$, i.e. groups of the form $\Int^n$ equipped with finitely generated rational metrics.

We define the proper subset of morphisms, resp. embeddings $\mathcal{M}$ on $\mathcal{K}$.
For $G,H\in \mathcal{G}$, the map $f:G\rightarrow H$ belongs to $\mathcal{M}$ if $f$ is an isometric group monomorphism that sends the algebraic free generators of $G$ to the algebraic free generators of $H$. In particular, $G$ is algebraically a direct summand of $H$, i.e. there exists some $G',F\in \mathcal{G}$ such that $H=G'\oplus F$ and moreover, $G$ is isometrically isomorphic to $G'$.
\end{defin}
\begin{prop}\label{isFraisse}
$\mathcal{G}$ is a $\sqsubseteq$-Fra\" iss\' e class.
\end{prop}
\begin{proof}
It is clear that $\mathcal{G}$ is countable, resp. it contains only countably many isomorphism types.

The $\sqsubseteq$-joint embedding property is a special case of the $\sqsubseteq$-amalgamation property which we will prove.

Suppose we are given $G_0,G_1,G_2\in \mathcal{G}$ such that there is an isometric homomorphism $\iota_i:G_0\hookrightarrow G_i$ for $i\in\{1,2\}$. To simplify the notation we will suppose that $G_i=G_0\oplus F_i$, for some $F_i\in \mathcal{G}$ and for $i\in\{1,2\}$; thus $\iota_1,\iota_2$ are just (isometric) inclusions. We define $G_3$ as $G_0\oplus F_1\oplus F_2$, $\rho_i:G_i\hookrightarrow G_3$, for  $i\in\{1,2\}$, is again just an inclusion. Suppose that $A_i$ generates the metric $d_i$ on $G_i$,  $i\in\{1,2\}$. We set $A_3=A_1\cup A_2$. For each pair $(a,b)\in A_3$ we define $$d'(a,b)=\begin{cases} d_1(a,b) & \text{if }(a,b)\in A_1,\\
d_2(a,b) & \text{if }(a,b)\in A_2.\end{cases}$$ Note that in case $(a,b)\in A_1\cap A_2$ we have that $(a,b)\in G_0^2$, thus $d_1(a,b)=d_2(a,b)$, and so there is no ambiguity in the definition.

Then for every $a,b\in G_3$ we set $$d_3(a,b)=\min\{d'(a_1,b_1)+\ldots+d'(a_n,b_n):$$ $$a=a_1+\ldots+a_n,b=b_1+\ldots+b_n, (a_i,b_i)\in A_3 \forall i\leq n\}.$$ What remains to prove is that the inclusions $\rho_1,\rho_2$ are indeed isometric.

We prove it for $\rho_1$. Let $a,b\in G_1$ be arbitrary, we must prove that $d_3(a,b)=d_1(a,b)$. The inequality $d_3(a,b)\leq d_1(a,b)$ is clear since if $d_1(a,b)=d_1(a_1,b_1)+\ldots+d_1(a_n,b_n)$ where $(a_i,b_i)\in A_1$ for all $i\leq n$, then also $(a_i,b_i)\in A_3$ and $d'(a_i,b_i)=d_1(a_i,b_i)$ for all $i\leq n$. Thus we have to prove the reverse inequality.

Suppose that $d_3(a,b)=d'(a_1,b_1)+\ldots+d'(a_n,b_n)$, where  $(a_j,b_j)\in A_3$ for all $j\leq n$. Since the group addition is commutative we can assume without loss of generality that there is $1\leq m\leq n$ such that for every $j\leq m$ we have $(a_j,b_j)\in A_2$ and for every $m<j\leq n$ we have $(a_j,b_j)\in A_1$.
 
We claim that $$a_1+\ldots+ a_m,b_1+\ldots+ b_m\in G_0.$$ To see this, observe that clearly $a_1+\ldots+ a_m,b_1+\ldots+ b_m\in G_2$. We also have $a_1+\ldots+ a_m,b_1+\ldots+ b_m\in G_1$. This holds since $a_{m+1}+\ldots+ a_n,b_{m+1}+\ldots+ b_n\in G_1$ and $a=a_1+\ldots+a_n,b=b_1+\ldots+b_n\in G_1$ and obviously $a_1+\ldots+ a_m=a-(a_{m+1}+\ldots+ a_n)$ and $b_1+\ldots+ b_m=b-(b_{m+1}+\ldots+ b_n)$. However, since $G_1\cap G_2=G_0$ the claim follows.

We then have $$d_3(a,b)=d_3(a_1+\ldots+ a_m,b_1+\ldots+ b_m)+$$ $$d_3(a_{m+1}+\ldots+ a_n,b_{m+1}+\ldots+ b_n).$$ By the arguments above and since the inclusions $\iota_1,\iota_2$ are isometries we have that $d_2(a_1+\ldots+ a_m,b_1+\ldots+ b_m)=d_0(a_1+\ldots+ a_m,b_1+\ldots+ b_m)=d_1(a_1+\ldots+ a_m,b_1+\ldots+ b_m)$ and thus $$d_3(a,b)=d_0(a_1+\ldots+ a_m,b_1+\ldots+ b_m)+$$ $$d_1(a_{m+1}+\ldots+ a_n,b_{m+1}+\ldots+ b_n)=d_1(a_1+\ldots+ a_m,b_1+\ldots+ b_m)+$$ $$d_1(a_{m+1}+\ldots+ a_n,b_{m+1}+\ldots+ b_n)=d_1(a,b)$$ and we are done. This finishes the proof of the lemma.
\end{proof}
It follows that $\mathcal{G}$ has a limit, denoted as $G$, which is a countable abelian group equipped with an invariant metric $d$. Algebraically, it is just $\bigoplus_{n\in \Nat} \Int$, the infinite direct sum of countably many copies of $\Int$, in other words, the free abelian group of countably many generators. It follows that the group operations on $G$ are continuous with respect to the topology induced by the metric.

The following property of $G$ implied by the Fra\" iss\' e theorem characterizes $G$ uniquely up to an isometric isomorphism.
\begin{fact}\label{charac_G}
Let $F$ be a finitely generated subgroup of $G$ such that $F\sqsubseteq G$. Let $H\in \mathcal{G}$ be such that there is an isometric homomorphism $\iota: F\hookrightarrow H$ such that $\iota[F]\sqsubseteq H$. Then there exists an isometric homomorphism $\rho:H\hookrightarrow G$ such that $\rho\circ \iota=\mathrm{id}_F$ and $\rho[H]\sqsubseteq G$.
\end{fact}

Consider the metric completion, denoted by $\mathbb{G}$, of $G$. It is a separable complete metric space and the group operations extend to the completion. It follows that $\mathbb{G}$ is a Polish abelian group equipped with an invariant metric. We refer the reader to \cite{Gao} for an exposition on Polish (metric) groups.\\
\begin{prop}\label{is_Ury}
$\mathbb{G}$ is isometric to the Urysohn universal space.
\end{prop}
To prove it, it suffices to prove that $G$, the countable dense subgroup of $\mathbb{G}$, is isometric to the rational Urysohn space. The following well known fact characterizes the rational Urysohn space.
\begin{fact}\label{char_Urys}
Let $(X,d)$ be a countable metric space with rational metric. Then it is isometric with the rational Urysohn space iff for every finite subset $A\subseteq X$ and for every rational Kat\v etov function $f:A\rightarrow \Rat^+$, where $\Rat^+$ denotes the set of positive rationals, there exists $x\in X$ such that for every $a\in A$, $(d(a,x)=f(a))$.
\end{fact}
So in order to prove Proposition \ref{is_Ury} it is sufficient to prove the following lemma. 
\begin{lem}\label{Kat_ext}
$(G,d)$ satisfies the condition from Fact \ref{char_Urys}.
\end{lem}
\begin{proof}[Proof of the lemma.]
Let $A\subseteq G$ be a finite subset and let $f:A\rightarrow \Rat^+$ be a Kat\v etov function.  Let $F_A$ be a finitely generated subgroup containing $A$ such that $F_A\sqsubseteq G$. We use Proposition \ref{Katetov_realiz} to obtain a group $F_A\oplus \Int\in \mathcal{G}$, where we denote the generator of the new copy of $\Int$ as $z$, such that $z$ realizes the Kat\v etov function $f$. However, using Fact \ref{charac_G} we immediately get that there exists $g\in G$ realizing $f$ and we are done.
\end{proof}
It was already mentioned in the introduction that the Shkarin/Niemiec's group with its metric denoted here as $(G_\mathrm{SN},d_\mathrm{SN})$ and any Cameron-Vershik group with its metric (Cameron-Vershik's groups form a family of continuum many groups) denoted here as $(G_\mathrm{CV},d_\mathrm{CV})$ are non-periodic (having elements of infinite order) abelian groups isometric to the Urysohn universal space. It is still open whether Shkarin/Niemiec's group belongs to the Cameron-Vershik family of groups. We can prove the following proposition.
\begin{prop}
The group $(\mathbb{G},d)$, which is also a non-periodic abelian group isometric to the Urysohn universal space, is isometrically isomorphic neither to $G_\mathrm{SN}$ nor to $G_\mathrm{CV}$.
\end{prop}
\begin{proof}
The metrics $d_\mathrm{SN}$ and $d_\mathrm{CV}$ both share the property that for every $x\in G_\mathrm{SN}$ and every $y\in G_\mathrm{CV}$ we have $\lim_{n\to \infty} \frac{d_\mathrm{SN}(n\cdot x)}{n}=0$ and $\lim_{n\to \infty} \frac{d_\mathrm{CV}(n\cdot y)}{n}=0$. For the former, it is explicitly mentioned and proved in both \cite{Sk} and \cite{Nie1}. For the latter, let $x\in G_\mathrm{CV}$ be the element whose orbit $\{n\cdot x:n\in \Int\}$ is dense in $G_\mathrm{CV}$ and isometric to the rational Urysohn space (see \cite{CaVe} for details). It suffices to prove $\lim_{n\to \infty} \frac{d_\mathrm{SN}(n\cdot x)}{n}=0$. However, since the orbit is isometric to the rational Urysohn space the set $\{n\in \Nat: d_\mathrm{CV}(x,n\cdot x)<1\}$ must be infinite and the claim follows.

It is obvious that the group $(\mathbb{G},d)$ does not have this property, so the result follows.
\end{proof}
\section{The universality}
In this section we prove the main result that $\mathbb{G}$ is a metrically universal separable abelian group. Because of Proposition \ref{sums_dense} it is sufficient to prove that for every invariant metric $d$ on $\bigoplus_{n\in \Nat} \Int$ there exists an isometric homomorphism $\iota: (\bigoplus_{n\in \Nat} \Int,d)\hookrightarrow \mathbb{G}$. Indeed, let $(H,d)$ be an arbitrary separable metric abelian group. Using Proposition \ref{sums_dense} we obtain a supergroup $(H\oplus(\bigoplus_{n\in \Nat} \Int),\bar{d})$, where $\bar{d}$ extends $d$, such that $\bigoplus_{n\in \Nat} \Int$ is dense in $H\oplus(\bigoplus_{n\in \Nat} \Int)$. Then we find an isometric homomorphism $\iota: (\bigoplus_{n\in \Nat} \Int,\bar{d})\hookrightarrow \mathbb{G}$ and the closure $\overline{\iota[\bigoplus_{n\in \Nat} \Int]}$ will contain an isometric copy of $(H,d)$.
\begin{defin}
Consider the group $\Int^n$, for some $n\geq 1$, where we denote the generators by $z_1,\ldots,z_n$. For every $x\in\Int^n$ by $|x|$ we denote the value $|k_1|+\ldots+|k_n|$, where $x=k_1\cdot z_1+\ldots+k_n\cdot z_n$.
Let $d$ and $d'$ be invariant metrics on $\Int^n$. We define a distance between $d$ and $d'$ as follows: we set $$D(d,d')=\sup\{\frac{|d(x)-d'(x)|}{|x|}: x\in\Int^n\setminus\{0\}\}.$$ In the sequel, we shall often say that two metrics $d$ and $d'$ are $\varepsilon$-close to mean that $D(d,d')\leq \varepsilon$, for some $\varepsilon>0$.
\end{defin}

We shall use this notion even in situations where we have some copy of $\Int^n$ denoted as $S_1\oplus\ldots\oplus S_n$ generated by $s_1,\ldots,s_n$ with a metric $d_S$ and another copy of $\Int^n$ denoted as $T_1\oplus\ldots\oplus T_n$ generated by $t_1,\ldots,t_n$ with a metric $d_T$. In that case we say $(S_1\oplus\ldots\oplus S_n,d_S)$ and $(T_1\oplus\ldots\oplus T_n,d_T)$ are $\varepsilon$-close if they are $\varepsilon$-close under a canonical algebraic isomorphism sending $s_i$ to $t_i$, for every $i\leq n$, which will be always clear from the context.\\

Our goal now is to prove that for each $n\geq 1$ the space of invariant metrics on $\Int^n$ equipped with the distance $D$ is separable, and moreover, the finitely generated rational metrics form a countable dense subset.
\begin{prop}\label{goodapprox}
Let $\Int^n$, for $n\geq 1$, be equipped with an arbitrary invariant metric $d$. Then for any $\varepsilon>0$ there exists a finitely generated metric $d_F$ on $\Int^n$ such that $d$ and $d_F$ are $\varepsilon$-close.
\end{prop}
\begin{proof}
Let us denote the generators by $z_1,\ldots,z_n$. For every $x=k_1\cdot z_1+\ldots+k_n\cdot z_n\in\Int^n$, by $x_i$ we denote the integer coefficient $k_i$, for $i\leq n$. Moreover, as before by $|x|$ we denote the value $|x_1|+\ldots+|x_n|$.

{For every $k\in\Nat$, denote by $A_k$ the set $\{x\in\Int^n:\forall i\leq n (|x_i|\leq k)\}$. Set $$B_k=\sup_{x\in\Int^n}\min \{\frac{|x-l\cdot y|}{|x|}:l\in\Nat,y\in A_k,|l\cdot y|\leq 2|x|\}.$$}

\begin{claim}
The sequence $(B_k)_k$ converges and we have $\lim_{k\to\infty} B_k=0$.
\end{claim}

Suppose for a moment that the claim has been proved and we show how to finish the proof of Proposition \ref{goodapprox}. Let $M=\max\{d(z_i):i\leq n\}$. Now choose $m_0$ so that for every $m\geq m_0$, $B_m<\varepsilon/(4M)$. For every $x\in A_{m_0}$ let $L_x=\lim \frac{d(l\cdot x)}{l}$, and let $m_1\geq m_0$ be such that for every $x\in A_{m_0}$ we have that for every $m\geq m_1$, $|\frac{d(m\cdot x)}{m}-L_x|<\varepsilon/16$. Finally, let $m\geq m_1$ be such that for every $y\in A_{m_0}$ and every $m'<m_1$ we have $\frac{d(m'\cdot y)}{m}<\varepsilon/16$.

Now let $d_F$ be the finitely generated metric generated by the values of $d$ on $(m A_{m_0})^2$, where $m A_{m_0}=\{m'\cdot x:m'\leq m,x\in A_{m_0}\}$; i.e. for every $a,b\in \Int^n$ we have $d_F(a,b)=\min\{d(a_1,b_1)+\ldots+d(a_j,b_j):j\in \Nat,(a_i,b_i)\in (m A_{m_0})^2\wedge a=a_1+\ldots + a_n,b=b_1+\ldots + b_n,i\leq j\}$. By Remark \ref{defin_remark}, for every $x,y\in (mA_{m_0})$ we have $d_F(x,y)=d(x,y)$. We claim that $d_F$ is as desired.

Take $x\in\Int^n$. Let $l\in\Nat$ and $y\in A_{m_0}$ be such that $\frac{|x-l\cdot y|}{|x|}\leq B_{m_0}$ and $|l\cdot y|\leq 2|x|$. We have $$\frac{|d(x)-d_F(x)|}{|x|}\leq \frac{|d(l\cdot y)-d_F(l\cdot y)|}{|x|}+\frac{d(x-l\cdot y)+d_F(x-l\cdot y)}{|x|}.$$

It suffices to show that $\frac{|d(l\cdot y)-d_F(l\cdot y)|}{|x|}<\varepsilon/2$ and that $\frac{d(x-l\cdot y)+d_F(x-l\cdot y)}{|x|}<\varepsilon/2$.

First we do the former. If $l\leq m$ then we have $d(l\cdot y)-d_F(l\cdot y)=0$. If $l>m$, then write $l$ as $l'\cdot m_1+r$, where $r<m_1$. Then we have 
\begin{multline}\label{eq2}
\frac{d_F(l\cdot y)}{l}\leq \frac{d_F(l'm_1\cdot y)+d_F(r\cdot y)}{l}\leq \frac{d_F(m_1\cdot y)}{m_1}+\frac{d_F(r\cdot y)}{m}=\\
\frac{d(m_1\cdot y)}{m_1}+\frac{d(r\cdot y)}{m}\leq L_y+\varepsilon/16+\varepsilon/16.
\end{multline}
Also, since $l>m\geq m_1$ we get
\begin{equation}\label{eq3}
|\frac{d(l\cdot y)}{l}-L_y|\leq \varepsilon/16.
\end{equation}

Note also that since $d_F$ agrees with $d$ on $m A_{m_0}$ we get by Lemma \ref{greatestmet_lem} that $d_F\geq d$, thus in particular 
\begin{equation}\label{simpleeq}
\frac{d_F(l\cdot y)}{l}\geq L_y-\varepsilon/16.
\end{equation}
Since $2|x|\geq l\cdot |y|$ we get from \eqref{eq2}, \eqref{eq3} and \eqref{simpleeq} that indeed $\frac{|d(l\cdot y)-d_F(l\cdot y)|}{|x|}<\varepsilon/2$. Now we do the latter. We have $\frac{d(x-l\cdot y)}{|x|}\leq \frac{|x-l\cdot y|M}{|x|}\leq B_{m_0}M$ as well as $\frac{d_F(x-l\cdot y)}{|x|}\leq \frac{|x-l\cdot y|M}{|x|}\leq B_{m_0}M$. Since $B_{m_0}<\varepsilon/(4M)$ we are done.\\

It remains to prove the claim. Fix some $\varepsilon>0$. We shall find $k\in\Nat$ such that $B_k<\varepsilon$ (it will be clear, though not needed, that for any $l>k$, $B_l<\varepsilon$). We start with some general estimates. For an arbitrary $k\in\Nat$ take some $x\in\Int^n\setminus A_k$. Denote by $|x|_\infty$ the value $\max_{i\leq n} |x_i|$ and by $l$ the value $\lceil \frac{|x|_\infty}{k}\rceil$. Then we can find $y^x\in A_k$ such that $|x-l\cdot y^x|\leq k+(n-1)l$, and so also $l\cdot |y^x|\leq 2|x|$. Indeed, let $i\leq n$ be such that $|x_i|=|x|_\infty$. Then we can take $y^x\in A_k$ such that $y^x_i=k$ if $x_i>0$, resp. $y^x_i=-k$ if $x_i<0$, and we have $|x_i-ly^x_i|<k$. For $j\leq n$ such that $j\neq i$ we can find $y^x_j\in [-k,k]$ such that $|ly^x_j|\leq 2|x_j|$ and $|x_j-ly^x_j|<l$. Now we have 
\begin{equation}\label{eq4}
\frac{|x-l\cdot y^x|}{|x|}\leq \frac{k+(n-1)(|x|_\infty/k+1)}{|x|_\infty}\leq \frac{k+n-1}{|x|_\infty}+\frac{n-1}{k}.
\end{equation}

Now we find suitable $k$. First, fix $k'$ large enough so that $(n-1)/k'<\varepsilon/2$. Next take $k$ large enough so that $(k'+n-1)/k<\varepsilon/2$. We now claim that $B_k<\varepsilon$. Take $x\in\Int^n$. We need to find $y\in A_k$ and $l$ such that $|l\cdot y|\leq 2|x|$ and $|x-l\cdot y|/|x|<\varepsilon$. If $|x|_\infty\leq k$ then $x\in A_k$ and we take $y=x$ and $l=1$. Otherwise, $|x|_\infty>k$ and we take $y^x\in A_{k'}\subseteq A_k$ and $l=\lceil \frac{|x|_\infty|}{k'}\rceil$ as above. By \eqref{eq4} we have $$\frac{|x-l\cdot y^x|}{|x|}\leq \frac{k'+n-1}{|x|_\infty}+\frac{n-1}{k'}<\varepsilon/2+\varepsilon/2.$$

This finishes the proof of the claim and Proposition \ref{goodapprox}.
\end{proof}
\begin{remark}\label{approx_remark}
Note that $d_F$ from the proof was generated by values of $d$ on $m A_k$ for some $k,m$. Thus by Lemma \ref{greatestmet_lem} we have $d_F\geq d$ since $d_F$ is the greatest invariant metric that agrees with $d$ on $m A_k$. Moreover, if $k'>k$ and $m'>m$, and $d'_F$ is generated by values of $d$ on $m' A_{k'}$, then again by Lemma \ref{greatestmet_lem}, $d\leq d'_F\leq d_F$. In particular, when choosing finitely generated metrics $d_F^\varepsilon$, resp. $d_F^{\varepsilon'}$, that are $\varepsilon$-close, resp. $\varepsilon'$-close to $d$, for $\varepsilon'<\varepsilon$, we may suppose that $d\leq d_F^{\varepsilon'}\leq d_F^\varepsilon$.
\end{remark}
\begin{lem}\label{realtorat}
Let $d$ be a finitely generated metric on $\Int^n$ and let $\varepsilon>0$. Then there exists a finitely generated rational metric $d_R$ on $\Int^n$ such that $d$ and $d_R$ are $\varepsilon$-close and $d\leq d_R$.
\end{lem}
\begin{proof}
Let us again denote the generators of $\Int^n$ as $z_1,\ldots,z_n$ and let $A\subseteq (\Int^n)^2$ be the generating set for $d$. We may assume that $(z_i,0)\in A$ for every $i\leq n$. Let $M=\max\{d(z_i,0):i\leq n\}$ and $m=\min\{d(a,b):a\neq b, (a,b)\in A\}$, and set $K=\lceil \frac{M}{m}\rceil$. For every $(a,b)\in A$ we set $d'_R(a,b)$ to be an arbitrary element of $[d(a,b)+\varepsilon/(2K),d(a,b)+\varepsilon/K]\cap \Rat$. Then we define the rational metric $d_R$ by means of $\{d'_R(a,b):(a,b)\in A\}$, i.e. as usual, for any $a,b\in \Int^n$ we set $$d_R(a,b)=\min\{d'_R(a_1,b_1)+\ldots+d'_R(a_l,b_l):$$ $$a=a_1+\ldots+a_l,b=b_1+\ldots+b_l,(a_i,b_i)\in A\; \forall i\leq l\}.$$ Note that clearly $d\leq d_R$.

Now let $k_1,\ldots,k_n\in \Int$ be arbitrary. To show that $d_R$ and $d$ are $\varepsilon$-close, we must check that $|d(k_1\cdot z_1+\ldots+k_n\cdot z_n)-d_R(k_1\cdot z_1+\ldots+k_n\cdot z_n)|<\varepsilon\cdot (|k_1|+\ldots+|k_n|)$. Suppose that $d(k_1\cdot z_1+\ldots+k_n\cdot z_n)=d(k_1\cdot z_1+\ldots+k_n\cdot z_n,0)=\sum_{i=1}^l d(a_i,b_i)$ where for every $i\leq l$ we have $(a_i,b_i)\in A$ and $k_1\cdot z_1+\ldots+k_n\cdot z_n=a_1+\ldots+a_l$, $0=b_1+\ldots+b_l$. Let $I=\{i\leq l:a_i\neq b_i\}$ and $l'=|I|$.

We claim that $l'\leq (|k_1|+\ldots+|k_n|)\cdot K$. Indeed, we have $$l'\cdot m\leq\sum_{i\in I} d(a_i,b_i)=\sum_{i=1}^l d(a_i,b_i)=d(k_1\cdot z_1+\ldots+k_n\cdot z_n)\leq$$ $$|k_1|\cdot d(z_1,0)+\ldots+|k_n|\cdot d(z_n,0)\leq (|k_1|+\ldots+|k_n|)\cdot M$$ which gives the inequality.

Thus we have $$d_R(k_1\cdot z_1+\ldots+k_n\cdot z_n)\leq\sum_{i=1}^l d'_R(a_i,b_i)=\sum_{i\in I} d'_R(a_i,b_i)=$$ $$\sum_{i\in I} (d(a_i,b_i)+\delta_i)\leq d(k_1\cdot z_1+\ldots+k_n\cdot z_n)+l'\cdot \varepsilon/K$$ where $\delta_i\in [\varepsilon/(2K),\varepsilon/K]$, thus we obtain $$d_R(k_1\cdot z_1+\ldots+k_n\cdot z_n)-d(k_1\cdot z_1+\ldots+k_n\cdot z_n)\leq l'\cdot\varepsilon/K\leq (|k_1|+\ldots+|k_n|)\cdot\varepsilon$$ which is what we wanted to prove. 
\end{proof}

\noindent
{\bf\underline{The embedding construction}}\\
We now start the embedding construction which will be done by an inductive process. Let $p$ be an arbitrary invariant metric on $\bigoplus_{n\in \Nat} \Int$. For notational reasons, we shall denote this group as $(\bigoplus_{n\in \Nat} Z_n,p)$ where $Z_n$, for each $n$, is a copy of $\Int$ generated by an element $z_n$. Recall from the beginning of this section that it suffices to embed only such groups. Also recall that $G$ denotes the countable dense subgroup of $\mathbb{G}$ constructed by the Fra\" iss\' e construction.

Before starting the embedding construction, let us define a sequence $(\rho_n)_{n\in\Nat}$ of positive reals (with the property that $\sum_n \rho_n<\infty$) as follows: for any $n\in\Nat$ we set $\rho_n=\min\{1/2^n,\min\{2\cdot p(z_i,0):i\leq n\}\}$.

\begin{claim}\label{claim1}
For every $n\in \Nat$ there is a subgroup $S_1^n\oplus\ldots\oplus S_n^n\sqsubseteq G$ with generators $s_1^n,\ldots,s_n^n$ such that $d\upharpoonright (S_1^n\oplus\ldots\oplus S_n^n)$ is $1/2^n$-close to $p\upharpoonright Z_1\oplus\ldots\oplus Z_n$ and moreover, for every $n$ and $m$ we have $d(s_n^m,s_n^{m+1})=\rho_n$; thus the sequence $(s_n^j)_j$ is Cauchy.
\end{claim}
Once this claim is proved we are done. Indeed, for every $n\in\Nat$ let $s_n\in \mathbb{G}$ be the limit of $(s_n^j)_j$ and let $S_n\leq \mathbb{G}$ denote the group generated by $s_n$. It follows that $\bigvee _{n\in\Nat} S_n\leq \mathbb{G}$ is isometrically isomorphic to $(\bigoplus_{n\in \Nat} Z_n,p)$.

Claim \ref{claim1} is a corollary of the following claim.
\begin{claim}\label{claim2}
There exists a copy of $\Int$ denoted as $T_1^1$ generated by $t_1^1$ with a finitely generated rational metric $d_1$, i.e. $(T_1^1,d_1)\in \mathcal{G}$ such that $(T_1^1,d_1)$ is $1/2$-close to $(Z_1,p)$. Moroever, for any $n>1$ there exists a copy of $\Int^n$ denoted as $T_1^n\oplus\ldots\oplus T_n^n$, generated by $t_1^n,\ldots,t_n^n$ such that there is a finitely generated rational metric $d_n$ on $(T_1^{n-1}\oplus\ldots\oplus T_{n-1}^{n-1})\oplus (T_1^n\oplus\ldots\oplus T_n^n)$, thus $((T_1^{n-1}\oplus\ldots \oplus T_{n-1}^{n-1})\oplus (T_1^n\oplus\ldots\oplus T_n^n),d_n)\in \mathcal{G}$,  such that $(T_1^n\oplus\ldots\oplus T_n^n,d_n)$ is $\rho_n$-close to $(Z_1\oplus\ldots\oplus Z_n,p)$, $d_n\upharpoonright (T_1^{n-1}\oplus\ldots\oplus T_{n-1}^{n-1})=d_{n-1}\upharpoonright (T_1^{n-1}\oplus\ldots\oplus T_{n-1}^{n-1})$ and for every $i<n$ we have $d_n(t_i^{n-1},t_i^n)=\rho_{n-1}$.
\end{claim}
Assume for a moment that Claim \ref{claim2} has been proved. Then consider $(T_1^1,d_1)$ from the claim. Since it belongs to the class $\mathcal{G}$ due to Fact \ref{charac_G} there exists an isometric homomorphism $\iota_1: (T_1^1,d_1)\hookrightarrow G$ such that $\iota_1[T_1^1]\sqsubseteq G$. Denote $s_1^1=\iota_1(t_1^1)$ and $S_1^1=\iota_1[T_1^1]$. Let $n>1$ be given and suppose that for every $i<n$ we have produced $S_1^i\oplus\ldots\oplus S_i^i\sqsubseteq G$ where $S_1^i\oplus\ldots\oplus S_i^i=\iota_i[T_1^i\oplus\ldots\oplus T_i^i]$ where for $i>1$, $\iota_i:(T_1^{i-1}\oplus\ldots\oplus T_{i-1}^{i-1})\oplus (T_1^i\oplus\ldots\oplus T_i^i)\hookrightarrow G$.

We use Claim \ref{claim2} to obtain a finitely generated rational metric $d_n$ on $(T_1^{n-1}\oplus\ldots \oplus T_{n-1}^{n-1})\oplus (T_1^n\oplus\ldots\oplus T_n^n)$. Since $d_n\upharpoonright (T_1^{n-1}\oplus\ldots\oplus T_{n-1}^{n-1})=d_{n-1}\upharpoonright (T_1^{n-1}\oplus\ldots\oplus T_{n-1}^{n-1})$, using Fact \ref{charac_G} we can find $\iota_n: (T_1^{n-1}\oplus\ldots \oplus T_{n-1}^{n-1})\oplus (T_1^n\oplus\ldots\oplus T_n^n)\hookrightarrow G$ such that $\iota_n[(T_1^{n-1}\oplus\ldots \oplus T_{n-1}^{n-1})\oplus (T_1^n\oplus\ldots\oplus T_n^n)]\sqsubseteq G$ and $\iota_n\upharpoonright (T_1^{n-1}\oplus\ldots\oplus T_{n-1}^{n-1})=\iota_{n-1} \upharpoonright (T_1^{n-1}\oplus\ldots\oplus T_{n-1}^{n-1})$. For every $i\leq n$ we then denote $s_i^n=\iota_n(t_i^n)$ and $S_i^n=\iota_n[T_i^n]$ and we are done.

Thus it remains to prove Claim \ref{claim2}.\\

Consider a copy of $\Int$ denoted as $T_1^1$ generated by $t_1^1$. Using Proposition \ref{goodapprox} and Lemma \ref{realtorat} we define a finitely generated rational metric $d_1$ on $T_1^1$ so that $p\upharpoonright Z_1$ and $d_1$ are $\rho_1$-close, i.e. for every $k\in \Int$ we have $|d(k\cdot z_1)-d_1(k\cdot t_1^1)|\leq |k|\cdot \rho_1$.\\

Suppose that for every $i<n$ we have already found $T_1^i\oplus\ldots\oplus T_i^i$ with generators $t_1^i,\ldots,t_i^i$ and a finitely generated rational metric $d_i$ such that $(T_1^i\oplus\ldots\oplus T_i^i, d_i)$ is $\rho_i$-close to $(Z_1\oplus\ldots\oplus Z_i,p)$ and moreover, for every $j<n$ and $k<n-1$ we have $d_{k+1}(t_j^k,t_j^{k+1})=\rho_k$.

Consider a copy of $\Int^n$ denoted as $T_1^n\oplus\ldots\oplus T_n^n$ with generators $t_1^n,\ldots,t_n^n$. Using Proposition \ref{goodapprox} and Lemma \ref{realtorat} we define a finitely generated rational metric $d'_n$ on $T_1^n\oplus\ldots\oplus T_n^n$ so that $(Z_1\oplus\ldots\oplus Z_n,p)$ and $(T_1^n\oplus\ldots\oplus T_n^n,d'_n)$ are $\rho_n$-close, i.e. for every $k_1,\ldots,k_n\in \Int$ we have $|p(k_1\cdot z_1+\ldots +k_n\cdot z_n)-d'_n(k_1\cdot t_1^n+\ldots +k_n\cdot t_n^n)|\leq \rho_n\cdot(|k_1|+\ldots+|k_n|)$. Also, by Remark \ref{approx_remark} and Lemma \ref{realtorat} it follows that for every $k_1,\ldots,k_{n-1}\in \Int$ we have
\begin{multline}\label{better_approx}
p(k_1\cdot z_1+\ldots+k_{n-1}\cdot z_{n-1})\leq d'_n(k_1\cdot t_1^n+\ldots+k_{n-1}\cdot t_{n-1}^n)\leq \\ d_{n-1}(k_1\cdot t_1^{n-1}+\ldots+k_{n-1}\cdot t_{n-1}^{n-1}).
\end{multline}

We now extend the finitely generated rational metric $d'_n$ on  $T_1^n\oplus\ldots\oplus T_n^n$ to a finitely generated rational metric $d_n$ on $(T_1^{n-1}\oplus\ldots\oplus T_{n-1}^{n-1})\oplus (T_1^n\oplus\ldots\oplus T_n^n)$ so that $d_n$ and $d_{n-1}$ coincide on $T_1^{n-1}\oplus\ldots\oplus T_{n-1}^{n-1}$.

Let $d_n$ be the unique finitely generated rational metric generated by $d'_{n-1}=d_{n-1}\upharpoonright (T_1^{n-1}\oplus\ldots\oplus T_{n-1}^{n-1})$, $d'_n$ and by the distance $d_n(t_j^{n-1},t_j^n)=\rho_{n-1}$ for every $j\leq n-1$. More precisely, suppose that $A'_{n-1}\subseteq (T_1^{n-1}\oplus\ldots\oplus T_{n-1}^{n-1})^2$ is a finite set of pairs that generates $d'_{n-1}=d_{n-1}\upharpoonright (T_1^{n-1}\oplus\ldots\oplus T_{n-1}^{n-1})$ and that $A'_n\subseteq (T_1^n\oplus\ldots\oplus T_n^n)^2$ is a finite set of pairs that generates $d'_n$. We set $A_n=A'_{n-1}\cup A'_n\cup A_{\rho_{n-1}}$, where $A_{\rho_{n-1}}=\{(\varepsilon\cdot t_j^{n-1},\varepsilon\cdot t_j^n),(\varepsilon\cdot t_j^n,\varepsilon\cdot t_j^{n-1}):j\leq n-1,\varepsilon\in\{1,-1\}\}$, to be the set generating $d_n$, and for every pair $(a,b)\in A_n$ we set $$d_n(a,b)=\begin{cases} d'_{n-1}(a,b)(=d_{n-1}(a,b)) & (a,b)\in A'_{n-1},\\
d'_n(a,b) & (a,b)\in A'_n,\\
\rho_{n-1} & (a,b)\in A_{\rho_{n-1}}.\\ \end{cases}$$

We need to check that this definition is consistent, i.e. for every $(a,b)\in A_n$ and for every $(a_1,b_1),\ldots,(a_m,b_m)\in A_n$ such that $a=a_1+\ldots+a_m$ and $b=b_1+\ldots+b_m$ we indeed have $d_n(a,b)\leq d_n(a_1,b_1)+\ldots+d_n(a_m,b_m)$ (see Remark \ref{defin_remark}).

This is easy if $(a,b)$ is equal to $(\varepsilon\cdot t_j^{n-1},\varepsilon\cdot t_j^n)$ or $(\varepsilon\cdot t_j^n,\varepsilon\cdot t_j^{n-1})$ for some $j\leq n-1$, $\varepsilon\in \{1,-1\}$. To see that, assume otherwise that for some $j\leq n-1$ there is a sequence $(a_1,b_1),\ldots,(a_m,b_m)\in A_n$ such that $t_j^n=a_1+\ldots+a_m$, $t_j^{n-1}=b_1+\ldots+b_m$ and $d_n(a_1,b_1)+\ldots+d_n(a_m,b_m)<\rho_{n-1}$. Clearly, for each $i\leq m$ we have $(a_i,b_i)\in A'_{n-1}\cup A'_n$ as otherwise we would, by definition, have $d_n(a_i,b_i)=\rho_{n-1}$. Since the group is abelian we may without loss of generality assume that there is $1\leq m'\leq m$ such that for every $i\leq m'$ we have $(a_i,b_i)\in A'_{n-1}$, and for each $m'<i\leq m$ we have $(a_i,b_i)\in A'_n$. Denote by $a^{m'}$ the element $a_1+\ldots+a_{m'}$, and by $a_{m'}$ the element $a_{m'+1}+\ldots+a_m$. The elements $b^{m'}$ and $b_{m'}$ are defined analogously. Since $t_j^n=a^{m'}+a_{m'}$, we must have $a^{m'}=0$ and $a_{m'}=t_j^n$. Similarly, since $t_j^{n-1}=b^{m'}+b_{m'}$, we must have $b^{m'}=t_j^{n-1}$ and $b_{m'}=0$. However, then we have $\sum_{i=1}^{m'} d_n(a_i,b_i)=\sum_{i=1}^{m'} d'_{n-1}(a_i,b_i)\geq d'_{n-1}(a^{m'},b^{m'})=d'_{n-1}(0,t_j^{n-1})\geq p(0,z_j)$. Similarly, we have $\sum_{i=m'+1}^m d_n(a_i,b_i)=\sum_{i=m'+1}^m d'_n(a_i,b_i)\geq d'_n(a_{m'},b_{m'})=d'_n(t_j^n,0)\geq p(z_j,0)$. Thus we get that $d_n(a_1,b_1)+\ldots+d_n(a_m,b_m)\geq 2\cdot p(z_j,0)\geq \rho_{n-1}$, and that is a contradiction.\\

We show the proof for the case when $(a,b)\in A'_{n-1}$. The last case, when $(a,b)\in A'_n$ should be clear; alternatively, one can use the proof for $(a,b)\in A'_{n-1}$ which will work for the last case as well.

So let $(a,b)\in A'_{n-1}$ and $(a_1,b_1),\ldots,(a_m,b_m)\in A_n$ be such that $a=a_1+\ldots+a_m$ and $b=b_1+\ldots+b_m$. Since the group is abelian we may without loss of generality assume that there are $0\leq m_1\leq m_2\leq m$ such that for every $i\leq m_1$ we have $(a_i,b_i)\in A'_{n-1}$, for every $m_1<i\leq m_2$ we have $(a_i,b_i)\in A'_n$ and for every $m_2<i\leq m$ we have $(a_i,b_i)\in \{(\varepsilon\cdot t_j^{n-1},\varepsilon\cdot t_j^n),(\varepsilon\cdot t_j^n,\varepsilon\cdot t_j^{n-1}):j\leq n-1, \varepsilon\in\{1,-1\}\}$. Let us write $a^{m_1}=a_1+\ldots+a_{m_1}$, $a_{m_1}^{m_2}=a_{m_1+1}+\ldots+a_{m_2}$ and $a_{m_2}=a_{m_2+1}+\ldots+a_m$; the definitions of $b^{m_1}$, $b_{m_1}^{m_2}$ and $b_{m_2}$ are analogous. It follows that $a^{m_1},b^{m_1}\in T_1^{n-1}\oplus\ldots\oplus T_{n-1}^{n-1}$ and $a_{m_1}^{m_2},b_{m_1}^{m_2}\in T_1^n\oplus\ldots\oplus T_n^n$.\\

Now, for every $x\in (T_1^{n-1}\oplus\ldots\oplus T_{n-1}^{n-1})\oplus (T_1^n\oplus\ldots\oplus T_n^n)$ let us write $x=k_1^{n-1}(x)\cdot t_1^{n-1}+\ldots+k_{n-1}^{n-1}(x)\cdot t_{n-1}^{n-1}+k_1^n(x)\cdot t_1^n+\ldots+k_n^n(x)\cdot t_n^n$.

Let us make two crucial observations:\\

\noindent \underline{Observation 1}
\begin{itemize}
\item $a_{m_1}^{m_2}$ and $b_{m_1}^{m_2}$ both lie in $T_1^n\oplus\ldots\oplus T_{n-1}^n$, i.e. we have $k_n^n(a_{m_1}^{m_2})=k_n^n(b_{m_1}^{m_2})=0$.\\

This is clear since $k_n^n(a)=k_n^n(a^{m_1})=k_n^n(a_{m_2})=0$ and $k_n^n(b)=k_n^n(b^{m_1})=k_n^n(b_{m_2})=0$.
\end{itemize}
\noindent\underline{Observation 2}
\begin{itemize}
\item We have $|k_1^n(a_{m_1}^{m_2})|+\ldots+|k_{n-1}^n(a_{m_1}^{m_2})|+|k_1^n(b_{m_1}^{m_2})|+\ldots+|k_{n-1}^n(b_{m_1}^{m_2})|\leq m-m_2$.\\

Since for every $i\leq n$ we have $k_i^n(a)=k_i^n(b)=k_i^n(a^{m_1})=k_i^n(b^{m_1})=0$ and $a=a^{m_1}+a_{m_1}^{m_2}+a_{m_2}$, $b=b^{m_1}+b_{m_1}^{m_2}+b_{m_2}$, we must have $k_i^n(a_{m_1}^{m_2})=-k_i^n(a_{m_2})$ and $k_i^n(b_{m_1}^{m_2})=-k_i^n(b_{m_2})$ for every $i\leq n$. Since for every $i\leq n$ and $m_2<j\leq m$ we have $|k_i^n(a_j)|\leq 1$ and $|k_i^n(b_j)|\leq 1$, and also $\sum_{(i\leq n;m_2<j\leq m)} |k_i^n(a_j)|+|k_i^n(b_j)|=m-m_2$, the observation follows.
\end{itemize}
Using these observations and the fact that $d_{n-1}$ on $T_1^{n-1}\oplus\ldots\oplus T_{n-1}^{n-1}$ and $d'_n$ on $T_1^n\oplus\ldots\oplus T_{n-1}^n$ are $\rho_{n-1}$-close (which follows from the inductive assumption and \ref{better_approx}) we get that $$d_{n-1}(k_1^n(a_{m_1}^{m_2})\cdot t_1^{n-1}+\ldots+k_{n-1}^n(a_{m_1}^{m_2})\cdot t_{n-1}^{n-1},k_1^n(b_{m_1}^{m_2})\cdot t_1^{n-1}+\ldots+k_{n-1}^n(b_{m_1}^{m_2})\cdot t_{n-1}^{n-1})-$$ $$d'_n(k_1^n(a_{m_1}^{m_2})\cdot t_1^n+\ldots+k_{n-1}^n(a_{m_1}^{m_2})\cdot t_{n-1}^n,k_1^n(b_{m_1}^{m_2})\cdot t_1^n+\ldots+k_{n-1}^n(b_{m_1}^{m_2})\cdot t_{n-1}^n)\leq$$ $$(|k_1^n(a_{m_1}^{m_2})|+\ldots+|k_{n-1}^n(a_{m_1}^{m_2})|+|k_1^n(b_{m_1}^{m_2})|+\ldots+|k_{n-1}^n(b_{m_1}^{m_2})|)\cdot \rho_{n-1} \leq$$ $$(m-m_2)\cdot \rho_{n-1}.$$\\ Since $\sum_{i=m_2+1}^m d_n(a_i,b_i)=(m-m_2)\cdot \rho_{n-1}$ we get $$d_n(a_1,b_1)+\ldots+d_n(a_m,b_m)\geq d'_{n-1}(a^{m_1},b^{m_1})+$$ $$(d'_{n-1}(k_1^n(a_{m_1}^{m_2})\cdot t_1^{n-1}+\ldots+k_{n-1}^n(a_{m_1}^{m_2})\cdot t_{n-1}^{n-1}, k_1^n(b_{m_1}^{m_2})\cdot t_1^{n-1}+\ldots+k_{n-1}^n(b_{m_1}^{m_2})\cdot t_{n-1}^{n-1})$$ $$-(m-m_2)\cdot \rho_{n-1})+(m-m_2)\cdot \rho_{n-1}\geq$$ $$d'_{n-1}(a^{m_1}-k_1^n(b_{m_1}^{m_2})\cdot t_1^{n-1}-\ldots-k_{n-1}^n(b_{m_1}^{m_2})\cdot t_{n-1}^{n-1},$$ $$b^{m_1}-k_1^n(a_{m_1}^{m_2})\cdot t_1^{n-1}-\ldots-k_{n-1}^n(a_{m_1}^{m_2})\cdot t_{n-1}^{n-1})=d'_{n-1}(a,b),$$
where only the last equality is not clear and does not follow from the previous discussion. We prove it here. It suffices to prove that $$a=a^{m_1}-k_1^n(b_{m_1}^{m_2})\cdot t_1^{n-1}-\ldots-k_{n-1}^n(b_{m_1}^{m_2})\cdot t_{n-1}^{n-1}$$ and similarly that $$b=b^{m_1}-k_1^n(a_{m_1}^{m_2})\cdot t_1^{n-1}-\ldots-k_{n-1}^n(a_{m_1}^{m_2})\cdot t_{n-1}^{n-1}.$$ We stress here that it is not a typo and indeed we claim that $a=a^{m_1}-k_1^n(b_{m_1}^{m_2})\cdot t_1^{n-1}-\ldots-k_{n-1}^n(b_{m_1}^{m_2})\cdot t_{n-1}^{n-1}$ instead of $a=a^{m_1}-k_1^n(a_{m_1}^{m_2})\cdot t_1^{n-1}-\ldots-k_{n-1}^n(a_{m_1}^{m_2})\cdot t_{n-1}^{n-1}$.

Since $a, a^{m_1}\in T_1^{n-1}\oplus\ldots\oplus T_{n-1}^{n-1}$ and $a=a^{m_1}+a_{m_1}^{m_2}+a_{m_2}$, we write $$a_{m_2}=-a_{m_1}^{m_2}+x,$$ where $x=a-a^{m_1}\in T_1^{n-1}\oplus\ldots\oplus T_{n-1}^{n-1}$, while $a_{m_1}^{m_2}\in T_1^n\oplus\ldots\oplus T_{n-1}^n$. Analogously, we write $$b_{m_2}=-b_{m_1}^{m_2}+y,$$ where $y=b-b^{m_1}\in T_1^{n-1}\oplus\ldots\oplus T_{n-1}^{n-1}$, while $b_{m_1}^{m_2}\in T_1^n\oplus\ldots\oplus T_{n-1}^n$. Thus we need to prove that $$x=-k_1^n(b_{m_1}^{m_2})\cdot t_1^{n-1}-\ldots-k_{n-1}^n(b_{m_1}^{m_2})\cdot t_{n-1}^{n-1}$$ and $$y=-k_1^n(a_{m_1}^{m_2})\cdot t_1^{n-1}-\ldots-k_{n-1}^n(a_{m_1}^{m_2})\cdot t_{n-1}^{n-1}.$$

Denote by $\phi$ the automorphism of $T_1^{n-1}\oplus\ldots\oplus T_{n-1}^{n-1}\oplus T_1^n\oplus\ldots\oplus T_{n-1}^n$ which sends $t_j^n$ to $t_j^{n-1}$ and $t_j^{n-1}$ to $t_j^n$, for $j\leq n-1$. Observe that $\phi(a_{m_2})=b_{m_2}$ and conversely $\phi(b_{m_2})=a_{m_2}$. Indeed, that immediately follows from the fact that for $m_2<j\leq m$, we have $(a_j,b_j)\in  \{(\varepsilon\cdot t_j^{n-1},\varepsilon\cdot t_j^n),(\varepsilon\cdot t_j^n,\varepsilon\cdot t_j^{n-1}):j\leq n-1, \varepsilon\in\{1,-1\}\}$, and that $a_{m_2}=\sum_{j=m_2+1}^m a_j$, $b_{m_2}=\sum_{i=m_2+1}^m b_i$. Then since $\phi[T_1^{n-1}\oplus\ldots\oplus T_{n-1}^{n-1}]=T_1^n\oplus\ldots\oplus T_{n-1}^n$ and $\phi[T_1^n\oplus\ldots\oplus T_{n-1}^n]=T_1^{n-1}\oplus\ldots\oplus T_{n-1}^{n-1}$, it follows that $$\phi[-a_{m_1}^{m_2}]=y$$ and $$\phi[-b_{m_1}^{m_2}]=x.$$ Finally, since $$\phi[-b_{m_1}^{m_2}]=-k_1^n(b_{m_1}^{m_2})\cdot t_1^{n-1}-\ldots-k_{n-1}^n(b_{m_1}^{m_2})\cdot t_{n-1}^{n-1}$$ and $$\phi[-a_{m_1}^{m_2}]=-k_1^n(a_{m_1}^{m_2})\cdot t_1^{n-1}-\ldots-k_{n-1}^n(a_{m_1}^{m_2})\cdot t_{n-1}^{n-1},$$ we are done.

This finishes the proof of the claim and the whole embedding construction.

\section{Generalizations and open questions}
Let $\kappa>\aleph_0$ be an uncountable cardinal number such that $\kappa^{<\kappa}=\kappa$, i.e. $\sup_{\gamma<\kappa} \kappa^{\gamma}=\kappa$. Any inaccessible cardinal is an example. Also, for any cardinal $\lambda$ such that $\lambda^+=2^\lambda$, $\kappa=\lambda^+$ is another example. Thus, for example under GCH any isolated cardinal has this property.

For such cardinals the generalized Fra\" iss\' e theorem holds. Notice at first that $\kappa=\kappa^{\aleph_0}\geq \mathfrak{c}$, where $\mathfrak{c}$ is the cardinality of continuum.

We shall consider the class $\mathcal{G}_{\kappa}$ of all groups $F$ where
\begin{itemize}
\item algebraically, $F$ is an infinite direct sum $\bigoplus_{\alpha<\lambda} \Int$ where $\lambda<\kappa$; i.e. a free abelian group of $\lambda$-many generators,
\item $F$ is equipped with an invariant metric $d$ (arbitrary, not necessarily rational-valued and not necessarily generated by values on some proper subset of $F^2$),
\item embeddings between groups from $\mathcal{G}_{\kappa}$ again send free generators to free generators and are isometric.
\end{itemize}

Using the property that $\kappa^{<\kappa}=\kappa$, one can readily check that the cardinality of $\mathcal{G}_{\kappa}$ is $\kappa$ even though we consider arbitrary real-valued invariant metrics.

The amalgamation is proved analogously as in the proof of Proposition \ref{isFraisse}. Suppose that we are given groups $G_0,G_1,G_2\in \mathcal{G}_{\kappa}$ with invariant metrics $d_0,d_1,d_2$ such that $G_i$, for $i\in\{1,2\}$, is isometrically isomorphic to $G_0\oplus F_i$, where $F_i\in \mathcal{G}_{\kappa}$. Then the amalgam is algebraically $G_0\oplus F_1\oplus F_2$. For any $(g_0,f_1,f_2)\in G_0\oplus F_1\oplus F_2$ we define $d_3((g_0,f_1,f_2),0)$ as $$\inf\{d_1((g,f',0),0)+d_2((h,0,f''),0):(g_0,f_1,f_2)=(g,f',0)+(h,0,f'')\}.$$ It is straightforward to check that this is correct. We note that this is equivalent to how the amalgamation is done in the category of Banach spaces.\\

Thus we get that $\mathcal{G}_{\kappa}$ is a generalized Fra\" iss\' e class of size $\kappa$. It has some Fra\" iss\' e limit $G_\kappa$ which is alebraically a free abelian group of $\kappa$-many generators equipped with an invariant metric $d$. Let us denote by $\mathbb{G}_\kappa$ the completion of $G_\kappa$.

Before we state the characterization of $G_\kappa$ let us note that the symbol $G\sqsubseteq H$ again denotes that $G$ is a subgroup of $H$ and moreover, the free generators of $G$ are a subset of free generators of $H$, exactly as in the separable case.
\begin{fact}\label{char_Gkappa}
Let $F\in \mathcal{G}_{\kappa}$ be such that $F\sqsubseteq G$. Let $H\in \mathcal{G}_{\kappa}$ be such that there is an isometric monomorphism $\iota: F\hookrightarrow H$ such that $\iota[F]\sqsubseteq H$. Then there exists an isometric monomorphism $\rho:H\hookrightarrow G$ such that $\rho\circ \iota=\mathrm{id}_F$ and $\rho[H]\sqsubseteq G$.
\end{fact}

Now suppose that $H$ is an arbitrary group of size (or equivalently density or weight) at most $\kappa$ that is equipped with invariant metric $d_H$. Using Proposition \ref{sums_dense} as in the separable case, we may without loss of generality assume that $H$ has a dense subgroup which is algebraically free abelian group of $\kappa$-many generators. Enumerate these generators as $\{h_\alpha: \alpha<\kappa\}$ and denote by $H_\alpha$ the subgroup generated by the first $\alpha$-many generators for every $\alpha<\kappa$. Since each $H_\alpha$ belongs to $\mathcal{G}_{\kappa}$, by transfinite induction using Fact \ref{char_Gkappa} we can embed the dense subgroup of $H$ generated by all the free generators $\{h_\alpha: \alpha<\kappa\}$ into $G_\kappa$. It follows that $H$ lies in $\mathbb{G}_\kappa$.\\

We conclude with few remarks and questions.\\

\begin{remark}
During the review process of the paper we obtained as a corollary of more general results that the metrically universal second-countable abelian group from the main theorem is extremely amenable, i.e. every continuous action of this group on a compact Hausdorff space has a fixed point (see \cite{Do}).
\end{remark}

\noindent {\bf Question 1.} Does there exist a metrically universal abelian group of weight $\kappa$ for every uncountable cardinal $\kappa$?\\

\noindent {\bf Question 2.} We would like to promote the question mentioned in \cite{Sk} whether there exists a metrically universal second-countable Hausdorff topological group with left-invariant metric.\\

\bigskip
\noindent {\bf Acknowledgment.} This work was done during the trimester program ``Universality and Homogeneity" at the Hausdorff Research Institute for
Mathematics in Bonn. The author would therefore like to thank for the support and great working conditions there. The author is also grateful to Wies\l aw Kubi\' s for discussions on this topic and to the referee whose comments significantly helped to improve the presentation.

\end{document}